\documentclass{article}

\usepackage{comment}
\usepackage[russian,spanish,british]{babel}
\usepackage[utf8]{inputenc}
\usepackage{amsmath}
\usepackage{amsfonts}
\usepackage{amssymb}
\usepackage{amsthm}
\usepackage{graphicx}
\usepackage{fancyhdr}
\usepackage{makeidx}
\usepackage{tikz}
\usepackage{subfigure}
\usepackage{hyperref}
\usepackage{dsfont}
\usepackage{parskip}
\usepackage{aliascnt}
\usepackage[round]{natbib}
\usepackage{enumerate}
\usepackage{mathrsfs}
\hypersetup{    colorlinks=false,    pdfborder={0 0 0},}
\begingroup
    \makeatletter
    \@for\theoremstyle:=definition,remark,plain\do{%
        \expandafter\g@addto@macro\csname th@\theoremstyle\endcsname{%
            \addtolength\thm@preskip\parskip
            }%
        }
\endgroup

\usepackage{boxedminipage}
\makeindex

\DeclareMathAlphabet{\mathpzc}{OT1}{pzc}{m}{it}

\newcommand{\desc}{\alpha} 
\newcommand{\F}{\mathcal{F}}

\newcommand{\R}{\mathds{R}}
\newcommand{\Z}{\mathds{Z}}

\newcommand{\Ga}{G_{\desc}}
\newcommand{\hGa}{\widehat{G}_{\desc}}
\newcommand{\Ra}{R_{\desc}}
\newcommand{\Va}{V_{\desc}}

\newcommand{\hit}[1]{\mathpzc{h}_{#1}}
\newcommand{\ind}[1]{\mathds{1}_{#1}}

\newcommand{\g}{{g}}
\newcommand{\gi}{\tilde{g}}

\newcommand{\D}{\mathcal{D}}
\newcommand{\Da}{\mathcal{D}_{\desc}}
\newcommand{\I}{\mathds{R}}
\newcommand{\ea}[1]{e^{-\desc{#1}}}

\newcommand{\estados}{\R}

\newcommand{\Aa}{A_{\desc}}
\newcommand{\Ex}[2]{\E_{#1}\left(#2\right)}
\newcommand{\be}{\begin{equation}}
\newcommand{\ee}{\end{equation}}
\newcommand{\bd}{\begin{equation*}}
\newcommand{\ed}{\end{equation*}}

\def\P{\operatorname{\mathds{P}}}
\def\E{\operatorname{\mathds{E}}}
\def\erf{\operatorname{erf}}

\newtheorem{teo}{Theorem}[section]

\newtheorem{cor}[teo]{Corolary}
\newtheorem{lem}[teo]{Lemma}

\theoremstyle{definition}

\newtheorem{remarks}[teo]{Remarks}
\newtheorem{defn}[teo]{Definition}
\newtheorem{example}[teo]{Example}

\title{Optimal stopping for Markov processes \\ with positive jumps}
\author{Fabi\'an Crocce\footnote{Facultad de Ciencias, Universidad de la Rep\'ublica, e-mail: fcrocce@gmail.com} \ and \ Ernesto Mordecki\footnote{Facultad de Ciencias, Universidad de la Rep\'ublica, e-mail: mordecki@cmat.edu.uy}}
\date{}
\begin{document}
\maketitle

\begin{abstract}
Consider the discounted optimal stopping problem for a real valued Markov process with only positive jumps. 
We provide a theorem to verify that the optimal stopping region has the form $\{x\geq x^*\}$ for some critical threshold $x^*$, 
and a representation formula for the value function of the problem in terms of the Green kernel of the process, based on Dynkin's characterization of the value function as the least excessive majorant.  
As an application of our results, using the Fourier transform to compute the Green kernel of the process, 
we solve a new example:
the optimal stopping for a L\'evy-driven Ornstein-Uhlenbeck process used to model prices in electricity markets.
\end{abstract}
\section{Introduction}

Optimal stopping problems (OSP) play a central role in optimization and control of stochastic process.
One possible formulation of the problem is the following: 
a player has to choose the moment to collect a reward that evolves randomly in time driven by a stochastic process $X$, 
using only the incoming information generated by the process 
(the player knows the law of the process).
In mathematical terms, 
the discounted optimal stopping problem to be considered in this paper consists in:
given 
a stochastic process $X=\{X_t\colon t\geq 0\}$ that starts at $X_0=x$, 
a reward function $g\colon\R\to[0,\infty)$, 
and a discount $\alpha>0$,
find the value function 
$V_\alpha(x)$ and the optimal stopping rule $\tau^*$ that satisfy
\begin{equation}\label{eq:osp}
V_\alpha(x) = \sup_{\tau}\E_x{e^{-\alpha\tau} g(X_{\tau})}=\E_x{e^{-\alpha\tau^*} g(X_{\tau^*})},
\end{equation}
where $\tau$ ranges in the class of stopping times. 
It is clear that many alternative formulations of the problem have interest, 
the process can have discrete time,
the discount can be $\alpha=0$, $X$ can take values in larger/smaller spaces (like $\Z$ or $\R^d$), 
the payoff function can involve non-markovian expectations (as the $\sup$ or the integral of the process),
and many others.
In other words, an OSP is an optimization problem where
the variable lies in the set of stopping times, and the objective function is the expectation of a stopped reward.

The optimal stopping problem of continuous time stochastic processes 
has a rich theory and a large amount of interesting applications. 
A central reference in the field is the monograph by Albert N. \cite{shirBook1978}, that supported the research in the field
during many years, including several editions and translations since the initial russian edition (\cite{Sh1}). 
Many of the posterior contributions in the field are included in the joint monograph of  \cite{ps}.

Within the initial works of optimal stopping of continuous time stochastic processes, 
we mention the quickest detection problems by 
\cite{sh1961}, 
 the perpetual option pricing problem by \cite{mckean},
 and the classical problems solved by \cite{taylor}.
In more recent times, we mention, notably, the Russian Options by \cite{shepp93}. 
A more comprehensive discussion, including the markovian-martingale dichotomy to solve OSPs,
 and a large list of references can be found in \cite{ps}.




Once this description of the advances in the field, in the present paper
we ask whether it is possible to solve explicitly optimal stopping problems for processes with jumps.
Optimal stopping problems for processes with jumps 
have a shorter history, summarized up to certain point in the monograph by \cite{kyprianou2006introductory}.
This review is concerned with optimal stopping problems for L\'evy processes, including the solutions provided
by \cite{mordeckiOSPOLP} and \cite{novikov2007solution}.
Problems for spectraly one sided L\'evy processes
where considered by several authors, for instance
\cite{avram2004exit} and \cite{chan} in the framework of option pricing.

Going a step further, in the present paper we propose to solve explicitely one OSP for a Markov process with jumps (one sided)
more general than a L\'evy process. 
In this regard, although there are some papers where verification theorem are provided,
for instance \cite{mordecki-salminen} or \cite{ChristensenSalminenBao}, to the knowledge of the authors,
no explicit problems in this framework were solved.

The approach we propose is based on a combination of the Riez's representation of an excessive function and the inversion formula for the infinitesimal generator of the process. Under the corresponding technical hypothesis to be detailed further, the proposal provides an equation to find the value function $V_\alpha$, that comes as a closed formula to compute this function in two main situations: 
when the process is a diffusion (see \cite{cm,cm2}) and when we have one sided problems with one sided jumps of the Markov process,
the problem considered in this paper.
More precisely, we are able to express the value function in terms of an integral that depends on the Green kernel of the process for an optimal stopping problem whose optimal stopping region has the form $\{x\geq x^*\}$ with $x^*$ a critical threshold, for Markov processes that have only positive jumps. It should be noticed that, despite the simplification of the assumption that the process has only one-sided jumps, the case considered is the ``difficult'' case, as the stopped process has a positive overshot.


The departing point of our proposal, inscribed in the Markovian approach, 
is Dynkin's characterization of the optimal stopping problem solution. 
Dynkin's characterization \cite{dynkin:1963} states (in our framework) that, if the reward function is lower semi-continuous,  $V$ is the value function of the non-discounted optimal stopping problem with reward $g$ if and only if $V$ is the least excessive function such that $V(x)\geq g(x)$ for all $x\in \R$.  
Applying this result for $Y$, the killed process of $X$ at rate $\alpha$, 
we obtain that $V_{\alpha}$, the value function of the problem with discount $\alpha$, is characterized as the least $\alpha$-excessive majorant of $g$.

As we mentioned, the second step uses Riesz's decomposition of an $\alpha$-excessive function. We recall this decomposition also in our context
(see \cite{kw,kw1,dynkin:1969}). 
A function $u\colon \I\to \mathbb{R}$ is $\alpha$-excessive if and only if there exist a non-negative Radon measure $\mu$ and an $\alpha$-harmonic function such that
\begin{equation}
\label{eq:alphaexcessive}
 u(x)=\int_{-\infty}^{\infty}G_{\alpha}(x,y)\mu(dy) + \text{($\alpha$-harmonic function)}.
\end{equation}
Furthermore, the previous representation is unique. The measure $\mu$ is called the representing measure of $u$.


The third step is based on the fact that the resolvent and the infinitesimal generator of  a Markov process are inverse operators. 
Suppose that we can write 
\begin{equation}\label{eq:inversionV}
V_{\alpha}(x)=\int_{-\infty}^{\infty}(\alpha-L) V_{\alpha}(y)G_{\alpha}(x,dy),
\end{equation}
where $L$ is the infinitesimal generator.
Assuming that the stopping region has the form
 $\{x\geq x^*\}$, and taking into account that $V_{\alpha}$ is $\alpha$-harmonic in the continuation region and $V_{\alpha}=g$ in the stopping region we obtain as a suitable candidate to be the representing measure
\begin{equation}\label{eq:repMeasure}
\mu(dy)=
\begin{cases}
0,                            				&\text{ if $y<x^*$},\\
(\alpha-{L})g(y)m(dy),  			&\text{ if $y>x^*$},
\end{cases}
\end{equation}
where $m(dy)$ is some convenient measure.
Based on these considerations, we present a key verification result, that provides a way to solve new examples. 


The proposal detailed above can be inscribed in the \emph{representation approach}, 
initiated by \cite{salminen85}, 
that consists in the representation of the value function $\Va$ in \eqref{eq:osp}, 
--the least $\alpha$-excessive majorant of the payoff $g$--
as an integral of a kernel, with respect to a radon measure, the \emph{representing measure}.
While the kernel carries information about the process,  
the measure happens to carry information about the payoff function $g$ and the stopping region. 
\cite{salminen85} uses the Martin kernel of a diffusion, in consequence the representing measure 
is the distribution of an $h$-transform of the stopped process (see \cite{dynkin:1969}).
The article of \cite{mordecki-salminen} provides also a representation theorem for optimal stopping of Hunt processes,
but in this case in terms of the Green kernel. There, the support of the measure is proved to be the stopping region for the problem. In \cite{cm,cm2} the density of the measure was computed in terms of $g$ for diffusions.

%
The rest of the paper is as follows. In Section \ref{section:2} we provide the necessary preliminaries and present the main results.
In Section \ref{section:proof} we present the proof of the main results.
In Section \ref{section:application} we present an application where the underlying process is a L\'evy-driven Ornstein-Uhlenbeck process, used in electricity models \cite[see][]{electricity}. 

\section{Preliminaries and main result}\label{section:2}
  

Given a probability space $(\Omega,\F,\P)$, 
consider a real valued time-homogeneous, non-terminating strong Markov process 
$\{X_t\colon t\geq 0\}$ 
adapted to a filtration 
$\{\F_t\colon t\geq 0\}$, ($\F_t\subset \F$ for all $t\geq 0$) 
and probabilities $\{\P_x\colon x\in\R\}$.


The filtration $\{\F_t\}$ is assumed to be \emph{right-continuous}, i.e. for all $t\geq 0$, we have that 
$\F_t=\bigcap_{s>0}\F_{t+s}=:\F_{t^+}$.
It is also assumed that $\F_t$ is $\P_x$-complete for all $t$ and for all $x$.


Given the filtration $\{\F_t\}$, a random variable $\tau$ in $(\Omega,\F)$, taking values in $[0,\infty]$ and such that $\{\omega:\tau(\omega)<t\}\in \F_t$, for all $t\geq 0$ is known as a \emph{stopping time} with respect to the filtration $\{\F_t\}$.
Given a stopping time $\tau$ with respect to a filtration $\{\F_t\}$ of $\F$, the family of sets $\F_\tau$ defined by
\begin{equation*}
 \F_\tau:=\{A \in \F\colon \forall t\geq 0, A\cap\{\omega:\tau(\omega)\leq t\} \in \F_t \}
\end{equation*}
is a sub $\sigma$-algebra of $\F$.

A progressively measurable Markov process $X=(\{X_t\},\{\F_t\},\{\P_x\})$ is said to be \emph{left-quasi-continuous} if for any stopping time $\tau$ with respect to $\{\F_t\}$, the random variable $X_{\tau}$ is $\F_{\tau}$-measurable and for any non-decreasing sequence of stopping times $\tau_n \to \tau$
\begin{equation*}
 X_{\tau_n}\to X_{\tau} \quad (\P_x-a.s\ \text{in the set}\ \{\tau<\infty\}),
\end{equation*}
for all $x\in\estados$.

\begin{defn} \label{def:standardMarkov} 
With the assumptions above, a real valued c\`adl\`ag (continous from the right with limits from the left), 
left quasi-continuous strong Markov process $X$ 
is a \emph{standard Markov process}.
\end{defn}
For general reference about Markov processes we refer to \cite{dynkin:books,dynkin1969theorems,karlin-tay,rogers2000diffusions,blumenthal-ge,revuz}. 
The formulation above is taken from 
the PhD Thesis of Fabi\'an
\cite{tesis}. 
%
%

Given a standard Markov process $X$ and a Borel function $f\colon \estados \to \R$, 
we say that $f$ belongs to the domain $\D$ of the extended infinitesimal generator of $X$, if there exists an Borel function 
$Af\colon \estados\to \R$ such that  $\int_0^t|Af(X_s)|ds<\infty$ almost surely for every $t$, and
$$
f(X_t)-f(X_0)-\int_0^t Af(X_s) ds
$$
is a right-continuous martingale with respect to the filtration $\{\F_t\}$ and the probability $\P_x$, for every $x\in \estados$ 
\citep[see][chap. VII, sect. 1]{revuz}.

The $\desc$-Green kernel of the process $X$ is defined by
\bd
\Ga(x,H):=\int_0^{\infty} \ea{t} \P_x(X_t\in H)dt,
\ed
for $x\in \estados$ and a Borel set $H$,
and the resolvent is the operator $\Ra$ given by
\be \label{eq:resolvent}
\Ra f(x):=\int_0^\infty \ea{t}\Ex{x}{f(X_t)}dt,
\ee
which can be defined for all Borel functions such that the previous integral makes sense for all $x\in \estados$. Note that if, for instance
\bd
\int_0^\infty \ea{t}\E_x|f(X_t)|dt<\infty \quad (x\in \estados)
\ed
then, using Fubini's theorem we may conclude that
\be
\Ra f(x)=\int_{\estados} f(y)\Ga(x,dy). \label{eq:RafGa}
\ee
Considering $T_\desc$ as a random variable with an exponential distribution with parameter $\desc$ (i.e. $\P(T_\desc\leq t)=1-\ea{t}$ 
for $t\geq 0$) and independent of $X$, define the process $Y$ with state space $\estados \cup \{\Delta\}$ 
--where $\Delta$ is an isolated point-- by
\bd Y_t:=
\begin{cases}
X_t & \text{if $t<T_\desc$,}\\
\Delta &\text{else.}
\end{cases}
\ed
Given a function $f:\estados\to\R$, we extend its domain by considering $f(\Delta):=0$. Observe that $\Ex{x}{f(Y_t)}=\ea{t}\Ex{x}{f(X_t)}$. We call $Y$ \emph{the $\desc$-killed process} with respect to $X$.

Consider a function $f$ that belongs to the domain $\Da$ of the extended infinitesimal generator of the $\desc$-killed process $Y$. In this case, there is a function $\Aa f\colon \estados \to \R$ such that
\bd
f(Y_t)-f(Y_0)-\int_0^t \Aa f(Y_s) ds
\ed
is a right-continuous martingale with respect to the filtration $\{\F_t\}$ and the probability $\P_x$, for every $x\in \estados$. Bearing the equality $\Ex{x}{f(Y_t)}=\ea{t}\Ex{x}{f(X_t)}$ in mind, it can be seen that
\bd
\ea{t}f(X_t)-f(X_0)-\int_0^t \ea{s}\Aa f(X_s) ds
\ed
is also a right-continuous martingale with respect to the filtration $\{\F_t\}$ and the probability $\P_x$, for every $x\in \estados$. Then
\bd
\Ex{x}{\ea{t}f(X_t)}-f(x)-\Ex{x}{\int_0^t \ea{s}\Aa f(X_s) ds}=0.
\ed
From the previous equation, and assuming that for all $x\in \estados$
\begin{itemize}
\item $\lim_{t\to \infty}\Ex{x}{\ea{t}f(X_t)}=0$ and
\item $\Ex{x}{\int_0^{\infty} \ea{s}|\Aa f(X_s)| ds}<\infty$,
\end{itemize}
we obtain, by taking the limit as $t\to \infty$ and using Lebesgue dominated convergence theorem, that
\bd
f(x)=\int_0^{\infty} \ea{s}\Ex{x}{-\Aa f(X_s)} ds.
\ed
Note that the right-hand side of the previous equation is $\Ra(-\Aa f)(x)$; from this fact and the previous equation we obtain, by \eqref{eq:RafGa},
\be
\label{eq:invExtended}
f(x)=\int_{\estados} {-\Aa f(y)}\Ga(x,dy).
\ee

It can be proved that if the function $f$ belongs to $\D$, it also belongs to $\Da$ and $\Aa f=Af-\desc f$.


Given a standard Markov process $X$ and a stopping time $\tau$, if $f=\Ra h$, we have that (see e.g. \cite{dynkin:books}, Theorem 5.1 or \cite{karlin-tay} equation (11.36))
\be \label{eq:dynFormMarkov}
f(x)=\Ex{x}{\int_0^\tau \ea{t}h(X_t) dt}+\Ex{x}{\ea{\tau}f(X_\tau)}.
\ee
As we will see further on, this formula has an important corollary in the analysis of optimal stopping problems.
Observe that it can be written in terms of $\Aa f$, when $f\in\Da$ by 
\bd
\Ex{x}{\ea{\tau}f(X_\tau)}-f(x)=\Ex{x}{\int_0^\tau \ea{t} \Aa f(X_t)dt};
\ed
being its validity a direct consequence of the Doob's optional sampling  theorem.

Consider a standard Markov process $X$ with state space $\R$. 
Denote by $\Delta X_t=X_t-\lim_{t\to s^-}X_s$ the jump of the process at time $t$,
and assume that $\Delta X_t\geq 0$.

\subsection{Main results}

Consider the optimal stopping problem \eqref{eq:osp} for  a standard Markov process $X$ with positive jumps and a non-negative continuous reward function $g$. 
The results we obtain (under the corresponding regularity assumptions) have the following form.
For each $x\in\R$, define
$$
f(x):=(\alpha-L)g(x), 
\qquad
\Va(x):=\int_{x^*}^{\infty} f(y) \Ga(x,dy)\text{ for all $x\in\R$,}
$$
where $x^*$ is the solution of the equation
\begin{equation}\label{eq:root}
g(x^*)=\int_{x^*}^{\infty}f(y)G_\alpha(x^*,dy),
\end{equation}
and the two conditions hold:
\begin{align}
f(y)\geq 0			&\text{ for $y\geq x^*$}\\
\Va(x)\geq g(x)	&\text{ for all $x\leq x^*$.}
\end{align}
Then the optimal stopping time and value function of problem \eqref{eq:osp} are respectively
\begin{equation}\label{eq:solution}
\tau^*=\inf\{t\geq 0\colon X_t\geq x^*\},
\quad
\Va(x)=\int_{x^*}^{\infty} f(y) \Ga(x,dy).
\end{equation}
Our main result follows.
\begin{teo}
\label{teo:huntVerif} 
Consider the optimal stopping problem \eqref{eq:osp} for 
a standard Markov process $X$ with positive jumps and a non-negative continuous reward function $g$.
Suppose that there exists two functions $f,\gi:\R\mapsto \R$ and a point $x^*\in\R$ such that
\begin{itemize}
\item[\rm(i)]
$\gi(x)=\int_{-\infty}^{\infty}f(y)\Ga(x,dy).$
\item[\rm(ii)]
$\gi(x^*)=\int_{x^*}^{\infty}f(y)\Ga(x^*,dy).$
\item[\rm(iii)] 
$f(x)\geq 0$ for all $x\geq x^*$. 
\item[\rm(iv)]
$\int_{x^*}^{\infty} f(y) \Ga(x,dy)\geq g(x)$ for $x\geq x^*$.
\item[\rm (v)]
$\g(x)=\gi(x)$ for $x\geq x^*$. 
\end{itemize}
Then, the solution of the problem \eqref{eq:osp} is given by \eqref{eq:solution}.
\end{teo}

If the function $g$ belongs to the domain of the extended infinitesimal generator, 
the unknown function above has an explicit form $f=-\Aa g$, and we obtain the following result.
\begin{cor}\label{corollary}
\label{cor:mainHuntVerif} Consider a strong Markov process $X$ with no negative jumps
and a continuous positive reward function $g$ that belongs to $\Da$ with $f=-\Aa g=(\alpha-A)g$, satisfying
\begin{itemize}
\item[\rm (i)] $\lim_{t\to \infty}\Ex{x}{\ea{t}g(X_t)}=0$,
\item[\rm (ii)] $\Ex{x}{\int_0^{\infty} \ea{s}|f(X_s)| ds}<\infty$.
\end{itemize} 

Suppose that $x^*$ is a solution of
\begin{equation*}
g(x^*)=\int_{x^*}^\infty(\alpha-A)g(y)\Ga(x^*,dy),
\end{equation*}
such that $(\alpha-A)g(y) \geq 0$ for all $x\geq x^*$.
Define 
\begin{equation*}
\Va(x)=\int_{x^*}^\infty(\alpha-A)\g(y) \Ga(x,dy).
\end{equation*}

If $\Va(x) \geq \g(x)$ for all $x\leq x^*$ then the OSP \eqref{eq:osp} is right-sided, $x^*$ is an optimal threshold and $\Va$ is
the value function.
\end{cor}


\section{Proofs}\label{section:proof}

We begin with the proof of the Corollary.
\begin{proof}[Proof of Corollary \ref{corollary}]
Observe that, by the assumptions on $\g$, for all $x$ in $\I$,
\bd
g(x)=\int_{\I}(-\Aa \g(y))\Ga(x^*,dy).
\ed
holds. Then, all the hypotheses of Theorem \ref{teo:huntVerif} are fulfilled with $f(x):=-\Aa \g(x)$ and $\gi:=g$,  this result leading to the thesis.
\end{proof}
It remains to prove Theorem \ref{teo:huntVerif}.
We start with a useful result concerning the Green kernel of processes without negative jumps.
\begin{lem} \label{lem:Gsaltos}
Let $X$ be standard Markov process without negative jumps. If $z<x$ and $H$ is a Borel set such that  $y< z$ for all $y$ in $H$, then
$$\Ga(x,H)=\E_x{ \ea{\hit{z}}} \Ga(z,H),$$
in other words
the ratio between $\Ga(x,H)$ and $\Ga(z,H)$ is independent of $H$.
\end{lem}
\begin{proof}
Since the process does not have negative jumps every path hits any intermediate state to go from $x$ to $H$. In other words, we know that for any trajectory
beginning from $x$ and such that $X_t \in H$ there exists some $s<t$ satisfying
$X_s=x^*$; hence
\begin{align*}
\P_x(X_t \in H)&=\int_0^t  \P_{x}(X_{t}\in H|X_s=x^*)\P_x (\hit{x^*}\in ds)\\
&=\int_0^t  \P_{x^*}(X_{t-s}\in H)\P_x (\hit{x^*}\in ds).
\end{align*}
Using the previous formula we obtain that
 \begin{align*}
  \Ga(x,H)&=\int_0^\infty \ea{t}\P_x(X_t \in H) dt\\
&= \int_0^\infty \ea{t} \left(\int_0^t  \P_{x^*}(X_{t-s}\in H)\P_x (\hit{x^*}\in ds)\right) dt\\ 
&= \int_0^\infty   \left( \int_s^\infty \ea{t}\P_{x^*}(X_{t-s}\in H) dt\right) \P_x(\hit{x^*}\in ds),
\end{align*}
where in the last equality we have changed the integration order; for the integral on the right-hand side we have
\begin{align*}
\int_s^\infty \ea{t}\P_{x^*}(X_{t-s}\in H) dt&= \ea{s} \int_s^\infty \ea{(t-s)}\P_{x^*}(X_{t-s}\in H) dt\\
&=\ea{s} \int_0^\infty \ea{t}\P_{x^*}(X_{t}\in H) dt\\
&=\ea{s} \Ga(x^*,H),
\end{align*}
obtaining that
\begin{align*}
  \Ga(x,H)&= \Ga(x^*,H) \int_0^\infty \ea{s} \P_x(\hit{x^*}\in ds)\\
  &=  \Ga(x^*,H) \Ex{x}{\ea{\hit{x^*}}}
\end{align*}
to conclude the proof.
\end{proof}

\begin{lem}\label{lem:hunt}
Consider a standard Markov process  $X$ without negative jumps. 
Assume for all $x\in \I$
\begin{equation*}
 g(x)=\int_{\I}f(y)\Ga(x,dy)
\end{equation*}
and suppose $x^*$ is such that
\begin{equation*}
 g(x^*)=\int_{x^*}^\infty f(y)\Ga(x^*,dy).
\end{equation*}
Then for all $x>x^*$ we have
\begin{equation*}
 g(x)=\int_{x^*}^\infty f(y) \Ga(x,dy).
\end{equation*}
\end{lem}
\begin{proof}
First observe that, from the definition of $\g$ and the equation defining $x^*$ we conclude that
\begin{equation}
\label{eq:hip}
\int_{-\infty}^{x^*}f(y)\Ga(x^*,dy)=0.
\end{equation}
Using the definition of $\g$ we get
\begin{align*}
 g(x)&=\int_{-\infty}^{\infty}f(y) \Ga(x,dy)\\
&=\int_{-\infty}^{x^*}f(y)\Ga(x,dy)+\int_{x^*}^{\infty}f(y)\Ga(x,dy).
\end{align*}
It remains to prove that, if  $x>x^*$, the fist term on the right-hand side of the previous equation vanishes; to do this, consider $x>x^*$, by \autoref{lem:Gsaltos}, we deduce that 
$$H\mapsto \Ga(x,H) \quad \mbox{and} \quad H\mapsto \Ex{x}{\ea{\hit{x^*}}}\Ga(x^*,H)$$
are the same measure in $\{x\leq x^*\}$; therefore
\begin{align*}
\int_{-\infty}^{x^*}
f(y)\Ga(x,dy)=\Ex{x}{\ea{\hit{x^*}}} \int_{-\infty}^{x^*}f(y)\Ga(x^*,dy),
\end{align*}
and vanishes by equation \eqref{eq:hip}.
\end{proof}

\begin{proof}[Proof of Theorem \ref{teo:huntVerif}]
By hypothesis $f(y)$ is non-negative for $y>x^*$, then we have that $\Va$ is an $\desc$-excessive function. 
By the application of Lemma \autoref{lem:hunt} we deduce that $\Va(x)$ coincides with $\gi(x)$ for $x>x^*$, therefore also coincides with $\g$. 
By hypothesis we obtain that $\Va$ dominates $\g$ in $\{x\leq x^*\}$. So $\Va$ is a majorant of $\g$ and, by Dynkin's characterization of the value function, we conclude that 
\begin{equation}\label{eq:xx}
\Va(x) \geq \sup_{\tau} \Ex{x}{e^{-\desc \tau}\g(X_\tau)}.
\end{equation}
We need some other auxiliary results.
We use the notation $\hit{B}$ for the hitting time of $B$,
\bd
\hit{B}:=\inf\{t\geq 0: X_t\in B\}
\ed
\begin{lem} \label{harmonic} 
For any Borel set $B$, the Green kernel satisfies;
\bd
\Ga(x,H)=\Ex{x}{\ea{\hit{B}}\Ga(X_{\hit{B}},H )},
\ed
for all $x$ in $\R$ and for all Borel set $H\subseteq B$.

In other words, for every $x\in \estados$, both $\Ga(x,dy)$ and $\Ex{x}{\ea{\hit{B}}\Ga(X_{\hit{B}},dy)}$, are the same measure in $B$.
\end{lem}
\begin{proof}
For $x\in B$ the assertion is clearly valid, since $\hit{B}\equiv 0$. Let us consider $x\in \estados\setminus B$. 
	By the definition of $\Ga$ and some manipulation, we obtain
	\begin{align*}
		\Ga(x,H)&=\Ex{x}{\int_0^{\infty} \ea{t}\ind{H}(X_t)dt}= \Ex{x}{\int_0^{\infty} \ea{t}\ind{H}(X_t)dt\ \ind{\{\hit{B}<\infty\}} } \\
		&=\Ex{x}{\int_0^{\hit{B}} \ea{t}\ind{H}(X_t)dt \ \ind{\{\hit{B}<\infty\}}} \\
		& \qquad +\Ex{x}{\int_{\hit{B}}^{\infty} \ea{t}\ind{H}(X_t)dt \ \ind{\{\hit{B}<\infty\}}},
	\end{align*}
	where the second equality holds, because if $\hit{B}$ is infinite, then $X_t$ does not hit $B$, therefore, $\ind{H}(X_t)=0$  for all $t$. In the third equality we simply split the integral in two parts. 
Note that the first term on the right-hand side of the previous equality vanishes, since $\ind{H}(X_t)$ is $0$ when $X_t$ is out of $B$ for the previously exposed argument. It remains to be proven that the second term is equal to $\Ex{x}{\ea{\hit{B}} \Ga(X_{\hit{B}},H )}$; the following chain of equalities completes the proof:
	\begin{align*}
		\Ex{x}{\int_{\hit{B}}^{\infty} \ea{t}\ind{H}(X_t)dt}
		&=\Ex{x}{e^{-\desc \hit{B}}\int_{0}^{\infty} \ea{t}\ind{H}(X_{t+\hit{B}})dt }\\ 
		&=\Ex{x}{\ea{\hit{B}} \Ex{x}{ \int_{0}^{\infty} \ea{t}\ind{H}(X_{t+\hit{B}})dt \big| \F_{\hit{B}}}}\\ 
		&=\Ex{x}{\ea{\hit{B}} \Ex{X_{\hit{B}}}{\int_{0}^{\infty} \ea{t}\ind{H}(X_{t})dt}}\\ 
		&= \Ex{x}{\ea{\hit{B}} \Ga(X_{\hit{B}},H)};
	\end{align*}
the first equality is a change of variable; in the second one, we take the conditional expectation into the expected value, and consider that $\hit{B}$ is measurable with respect to $\F_{\hit{B}}$; the third equality is a consequence of the strong Markov property; while in the last one, we use the definition of $\Ga$.
\end{proof}

\begin{lem}
\label{lem:generalMin}
Let be given a Borel function $f\colon \R \to \R$ and a Borel set $B$
s.t. $\int_{B} |f(y)| \Ga(x,dy)<\infty$ for all $x\in\R$.
Denote by $F_B \colon\R\to \R$ the function 
\begin{equation*}
\label{eq:defF}
F_B(x):=\int_{B} f(y) \Ga(x,dy).
\end{equation*}
Then
\begin{equation*}
F_B(x)=\Ex{x}{\ea{\hit{B}}F_S(X_{\hit{B}})}.
\end{equation*}
\end{lem}
\begin{proof}
From \autoref{harmonic} we get that
$$
F_B(x)=\int_{B} f(y) \Ex{x}{e^{-\desc \hit{B}} \Ga(X_{\hit{B}},dy)}.
$$
Changing the integration sign with the expected value on the right-hand side of the equation, we complete the proof.
\end{proof}

To conclude that the other inequality in \eqref{eq:xx} also holds, we apply Lemma \ref{lem:generalMin} with $B:=\{x>x^*\}$ and $F_B:=\Va$ 
obtaining that 
\bd
\Va(x)=\Ex{x}{e^{-\desc \hit{B}}\Va(X_{\hit{B}})}.
\ed
Since the trajectories are right continuous, it gathers that $X_{\hit{B}}$ belongs to $\{x\geq x^*\}$, 
the region in which $\Va$ and $\g$ coincide; therefore
\bd
\Va(x)=\Ex{x}{e^{-\desc \hit{B}}g(X_{\hit{B}})},
\ed
proving 
\bd
\Va(x) \leq \sup_{\tau} \Ex{x}{e^{-\desc \tau}\g(X_\tau)}.
\ed
We have proved the desired equality concluding that the optimal stopping
problem is right-sided with threshold $x^*$.
\end{proof}


\section{Application: L\'evy-driven Ornstein-Uhlenbeck with positive jumps}\label{section:application}

Let $X$ be a L\'evy-driven Ornstein-Uhlenbeck process; i.e. a process satisfying the stochastic differential equation
\be \label{eq:orn-uhl-sde}
dX_t=-\gamma X_{t^-}dt+dL_t, \qquad X_0=x,
\ee
where $\{L_t\}$ is a L\'evy process.
The consideration of this process is motivated by its application to model electricity markets \cite[see][]{electricity}.
The only solution of the equation \eqref{eq:orn-uhl-sde} is \citep[see][]{novikov2006levyornstein}
\be \label{eq:orn-uhl-sde-sol}
X_t=e^{-\gamma t}\left(\int_0^t e^{\gamma s}dL_s+X_0\right).
\ee
In our example we consider $L_t= \sigma B_t+ J_t$ where $\{J_t\}$ is a compound Poisson process with rate $\lambda$ and positive jumps with exponential distribution of parameter $\beta$; i.e.
\bd
J_t=\sum_{i=1}^{N_t}Y_i,
\ed
with $\{N_t\}$ a Poisson process with rate $\lambda$ and $Y_i$ independent identically distributed random variables, with exponential distribution of parameter $\beta$. Observe that there are only positive jumps.

We aim to solve the optimal stopping problem \eqref{eq:osp} with reward function $g(x)=x^+$, i.e. to find the stopping time $\tau^*$ and the value function $\Va$ such that
\bd
\Va(x)=\Ex{x}{\ea{\tau^*}{X_{\tau^*}}^+}=\sup_{\tau}\left(\Ex{x}{\ea{\tau}{X_{\tau}}^+}\right).
\ed
In order to solve this problem we apply Theorem \ref{teo:huntVerif} with $\gi\colon \gi(x)=x$; hence we need to find $f$
satisfying (i). Consider the following equalities
\begin{align}\label{eq:itoO-U}
\ea{t}X_t-X_0&=\int_{(0,t]}X_{s^-}(-\desc \ea{s})ds+\int_{(0,t]}\ea{s}dX_s\\
\notag &=-\int_{(0,t]}X_{s^-} (\desc+\gamma) \ea{s}ds +\int_{(0,t]}\ea{s} \sigma dB_s+\int_{(0,t]}\ea{s} dJ_s.
\end{align}
The expected value of the integral with respect to $\{B_s\}$ vanishes. 
Concerning the integral with respect to the jump process 
$\{J_s\}$, we can write it in terms of a jump measure $\mu$ -defined by
\be \label{eq:defPoissRandMeasure}
\mu^X(\omega,dt,dx)=\sum_s \ind{\{\Delta X_s(\omega)\neq 0\}}\delta_{(s,\Delta X_s(\omega))}(dt,dx)
\ee
 See \citep[see][Proposition 1.16]{jacod1987limit}. 
 We obtain
\begin{align*}
\int_{(0,t]}\ea{s} dJ_s&=\int_{(0,t]\times \R}\ea{s} y \mu(\omega,ds,dy)\\
&=\int_{(0,t]\times \R}\ea{s} y (\mu(\omega,ds,dy)-\nu(ds,dy)) +\int_{(0,t]\times \R}\ea{s} y \nu(ds,dy),
\end{align*}
where $\nu$, the compensator of $\mu$, in this case is given by 
\bd
\nu(ds,dy)=\lambda\beta \ind{\{y>0\}}e^{-\beta y}dy ds.
\ed
From the application of Corollary 4.6 in \cite{kyprianou2006introductory}, it follows that
\bd
M_t=\int_{(0,t]\times \R}\ea{s} y (\mu(\omega,ds,dy)-\nu(ds,dy))
\ed
is a martingale, then $\Ex{x}{M_t}=\Ex{x}{M_0}=0$. It follow that
\begin{align*}
\Ex{x}{\int_{(0,t]}\ea{s} dJ_s}&=\int_{(0,t]\times \R}\ea{s} y \nu(ds,dy)\\
&=\int_{(0,t]}\int_{\R^+} \ea{s} y \lambda\beta e^{-\beta y}dy ds=\int_{(0,t]} \ea{s} \frac\lambda\beta  ds.
\end{align*}
Taking the expectation in \eqref{eq:itoO-U} we obtain that
\be \label{eq:itoO-U-expectation}
\Ex{x}{\ea{t}X_t}-x=-\Ex{x}{\int_{(0,t]}\left(X_{s^-} (\desc+\gamma)-\frac\lambda\beta \right)\ea{s}ds}.
\ee
Using \eqref{eq:orn-uhl-sde-sol} we compute $\Ex{x}{X_t}$:
\begin{align*}
\Ex{x}{X_t}&=\Ex{x}{e^{-\gamma t}\left(\int_0^t e^{\gamma s}dL_s+X_0\right)}\\
&= e^{-\gamma t}\left(\Ex{x}{\int_0^t e^{\gamma s}\sigma dB_s}+\Ex{x}{\int_0^t e^{\gamma s}\sigma dJ_s} + x\right)\\
&= (1-e^{-\gamma t})\frac\lambda{\beta\gamma}+ x e^{-\gamma t},
\end{align*}
concluding that $\lim_{t\to \infty} \ea{t}\Ex{x}{X_t}=0$. With similar arguments we obtain $\Ex{x}{|X_t|}\leq \frac{1}{\sqrt{\pi \gamma}} + \frac{\lambda t}{\beta}$. 
We can change the order between the expectation and the integral on the right-hand side of \eqref{eq:itoO-U-expectation}. Taking the limit as $t\to \infty$ in \eqref{eq:itoO-U-expectation} we obtain that
\bd
-x=-\int_0^\infty \Ex{x}{X_s(\desc+\gamma)-\frac\lambda\beta}\ea{s}ds.
\ed
The previous equality can be written in terms of the Green kernel by
\be \label{eq:inversionO-U}
x=\int_{\R} \left(y(\desc+\gamma)-\frac\lambda\beta\right)\Ga(x,dy)
\ee
which is (i) in Theorem \ref{teo:huntVerif} with 
\be \label{eq:fparainveOrnstein}
f(y)=y(\desc+\gamma)-\frac\lambda\beta.
\ee
Now we move on to find the Green kernel of the process.

It can be seen that for the considered process there exist a function $\Ga(x,y)$ such that $\Ga(x,dy)=\Ga(x,y)dy$. As we can not find $\Ga(x,y)$ explicitly we compute its Fourier transform,
\begin{align*}
\hGa(x,z)&=\int_{-\infty}^\infty e^{izy}\Ga(x,y)dy\\
&= \int_0^\infty \ea{t} \int_{-\infty}^\infty e^{izy} \P_x(X_t\in dy) dt\\
&= \int_0^\infty \ea{t} \Ex{x}{e^{izX_t}} dt.
\end{align*}
We need to compute $\Ex{x}{e^{izX_t}}$. In order to do that we apply Dynkin's formula to $u(x)=e^{izx}$. We have
\[u'(x)=iz u(x)\quad \text{and} \quad u''(x)=-z^2 u(x).\]
and
\begin{align*}
Lu(x)&=-\gamma x iz u(x)+\frac{-z^2}{2} u(x)+u(x) \lambda \beta \int_0^\infty \left(e^{izy}-1\right)e^{-\beta y}dy\\
&= u(x)\left(-\gamma xiz+\frac{-z^2}{2}+\frac{\lambda \beta}{\beta-i z}-\lambda \right)\\
&= u(x)\left(-\gamma xiz+\frac{-z^2}{2}+\frac{iz \lambda}{\beta-i z} \right).
\end{align*}
By Dynkin's formula we obtain that
\begin{align*}
\Ex{x}{e^{izX_t}}-e^{izx}&=\Ex{x}{\int_0^t u(X_s)\left(-\gamma X_s iz+\frac{-z^2}{2}+\frac{iz \lambda}{\beta-i z} \right) ds} 
\end{align*}
Denoting by $h(x,t,z)=\Ex{x}{e^{i z X_t}}=\Ex{x}{u(X_t)}$ we have 
\[h_z(x,t,z)=\Ex{x}{i X_t u(X_t)}\]
 and the previous equation is
\be \label{eq:hxtz}
h(x,t,z)-e^{izx}=\int_0^t -\gamma z h_z(x,s,z) + \left(-\frac{z^2}{2}+\frac{\lambda i z}{\beta-iz}\right)h(x,s,z)ds.
\ee
Instead of solving the previous equation we try to find directly 
$\hGa(x,z)$. Remember that
\begin{align*}
\hGa(x,z)&= \int_0^\infty \ea{t} h(x,t,z) dt.
\end{align*}
Taking Laplace transforms in \eqref{eq:hxtz} we obtain that
\begin{align*}
\hGa(x,z)-e^{izx}/\desc &= \int_0^\infty ds \int_s^\infty -\gamma z h_z(x,s,z) \ea{t}+ \left(-\frac{z^2}{2}+\frac{\lambda i z}{\beta-iz}\right)h(x,s,z)\ea{t}  dt\\
&= \frac{1}{\desc} \int_0^\infty -\gamma z h_z(x,s,z) \ea{s}+ \left(-\frac{z^2}{2}+\frac{\lambda i z}{\beta-iz}\right)h(x,s,z)\ea{s}  ds,
\end{align*}
which is equivalent to
\bd
\desc \hGa(x,z) - e^{izx} = -\gamma z \frac{\partial \hGa}{\partial z}(x,z)+\left(-\frac{z^2}{2}+\frac{\lambda i z}{\beta-iz}\right) \hGa(x,z)
\ed
and to 
\be
\label{eq:ecuacion}
 \left(\desc+\frac{z^2}{2}-\frac{\lambda i z}{\beta-iz}\right)  \hGa(x,z)+\gamma z \frac{\partial \hGa}{\partial z}(x,z)  = e^{izx}.
\ee
About the initial condition, observe that $\hGa$ satisfies
\begin{align*}
\hGa(x,0)&=\int_{-\infty}^\infty \Ga(x,dy)= \int_0^\infty \ea{t} \int_{-\infty}^\infty \P_x(X_t\in dy) dt\\
&= \int_0^\infty \ea{t} dt = \frac{1}{\desc}.
\end{align*}
We solve explicitly  \eqref{eq:ecuacion}: Let us start by solving the homogeneous equation
\bd
 \left(\desc+\frac{z^2}{2}-\frac{\lambda i z}{\beta-iz}\right)  H(z)+\gamma z H_z(z)  = 0,
\ed
obtaining that
\bd
\frac{H_z(z)}{H(z)}  = -\frac{\left(\desc+\frac{z^2}{2}-\frac{\lambda i z}{\beta-iz}\right)}{ \gamma z },
\ed
and
\bd
\log(H(z))=\frac{-\desc}{\gamma}\log(|z|)+\frac{-1}{4\gamma}z^2-\frac{\lambda}{\gamma}\log(\beta-iz),
\ed
then
\bd
H(z)=e^{-\frac{1}{4\gamma}z^2}|z|^{-\frac{\desc}{\gamma}}(\beta-iz)^{-\frac{\lambda}{\gamma}}.
\ed
The solution of \eqref{eq:ecuacion} is given by
\begin{align}
\notag \hGa(x,z)&=\frac{H(z)}{\gamma}\int_{0}^z  \frac{e^{i\zeta x}}{\zeta H(\zeta)}d\zeta\\ 
\label{eq:Ggorro} &=\frac{1}{\gamma}\left(e^{-\frac{1}{4\gamma}z^2}|z|^{-\frac{\desc}{\gamma}}(\beta-iz)^{-\frac{\lambda}{\gamma}}\right)\int_{0}^z  e^{i\zeta x}e^{\frac{1}{4\gamma}\zeta^2} \frac{|\zeta|^{\frac{\desc}{\gamma}}}{\zeta}(\beta-i\zeta)^\frac{\lambda}{\gamma} d\zeta
\end{align}
for $z\neq 0$.
Observe that $H(z)\to \infty$ as $z\to 0$, being equivalent to $|z|^{-\frac{\desc}{\gamma}}\beta^{-\frac{\lambda}{\gamma}}$. On the other hand
\bd
\int_{0}^z  \frac{e^{i\zeta x}}{\zeta H(\zeta)}d\zeta \to 0\quad (z\to 0)
\ed
since the integral is convergent. We can use l'H\^opital rule to compute the limit of $\hGa(x,z)$ when $z$ goes to 0. We obtain that
\begin{align*}
\lim_{z\to 0}\hGa(x,z)&=\lim_{z\to 0} \frac{H(z)}{\gamma}\int_{0}^z  \frac{e^{i\zeta x}}{\zeta H(\zeta)}d\zeta \\
&=\lim_{z\to 0} \frac{1}{\gamma}\left(|z|^{-\frac{\desc}{\gamma}}\beta^{-\frac{\lambda}{\gamma}}\right)\int_{0}^z  \frac{e^{i\zeta x}}{\zeta H(\zeta)}d\zeta \\
&=\lim_{z\to 0} \frac{1}{\gamma}\frac{\int_{0}^z  \frac{e^{i\zeta x}}{\zeta H(\zeta)}d\zeta }{|z|^{\frac{\desc}{\gamma}}\beta^{\frac{\lambda}{\gamma}}} \\
&=\lim_{z\to 0} \frac{1}{\gamma}\frac{e^{i z x}e^{\frac{1}{4\gamma} z^2} |z|^{\frac{\desc}{\gamma}-1}(\beta-i z)^\frac{\lambda}{\gamma} }{\frac{\desc}{\gamma} |z|^{\frac{\desc}{\gamma}-1}\beta^{\frac{\lambda}{\gamma}}}  =\frac{1}{\desc}
\end{align*}
concluding that the solution we found satisfies the initial condition.

We have obtained an expression for $\hGa(x,z)$, which allow us, for particular values of the parameters, to compute a discretization of $\hGa(x,z)$. From this discretization, using the discrete Fourier transform, we find a discretization of $\Ga(x,y)$ (we have written an R script to do this) necessary to solve equation (ii) in Theorem \ref{teo:huntVerif}.

\begin{example}[$\beta=\desc=\gamma=\lambda=1$]
Consider the process already presented with parameters $\beta=\desc=\gamma=\lambda=1$. Equation \eqref{eq:Ggorro} is 
\begin{align*}
\hGa(x,z)&=\left(e^{-\frac{1}{4}z^2}|z|^{-1}(\beta-iz)^{-1}\right)\int_{0}^z  e^{i\zeta x}e^{\frac{1}{4}\zeta^2}\frac{|\zeta|}{\zeta} (\beta-i\zeta) d\zeta\\
&=\left(e^{-\frac{1}{4}z^2}z^{-1}(\beta-iz)^{-1}\right) \\
&\qquad \left(i \sqrt{\pi} e^{x^2} (\beta-2x)\left(\erf\left(x-\frac{iz}{2}\right)-\erf(x)\right) 
-2 i (e^{izx+\frac{1}{4}z^2}-1)\right)
\end{align*}

Remember that we are considering the reward function $\g(x)=x^+$. 
To solve numerically equation (ii) in Theorem \ref{teo:huntVerif} we use: $\tilde{g}(x)=x$; function $f$ given in \eqref{eq:fparainveOrnstein}; and the discretization of $\Ga(x,y)$ obtained numerically as described above. The solution we found is $x^*= 1.1442$. \autoref{fig:ornstein-jumps} shows some points of the value function, obtained numerically by the formula
\bd
\Va(x)=\int_{x^*}^\infty \Ga(x,y)f(y) dy.
\ed
We also include in the plot the reward function (continuous line). Observe that for $x<x^*$ (in the continuation region) $\Va>g$ and the hypothesis of \autoref{teo:huntVerif} is fulfilled.
\begin{figure}
\begin{center}
\includegraphics[scale=.4]{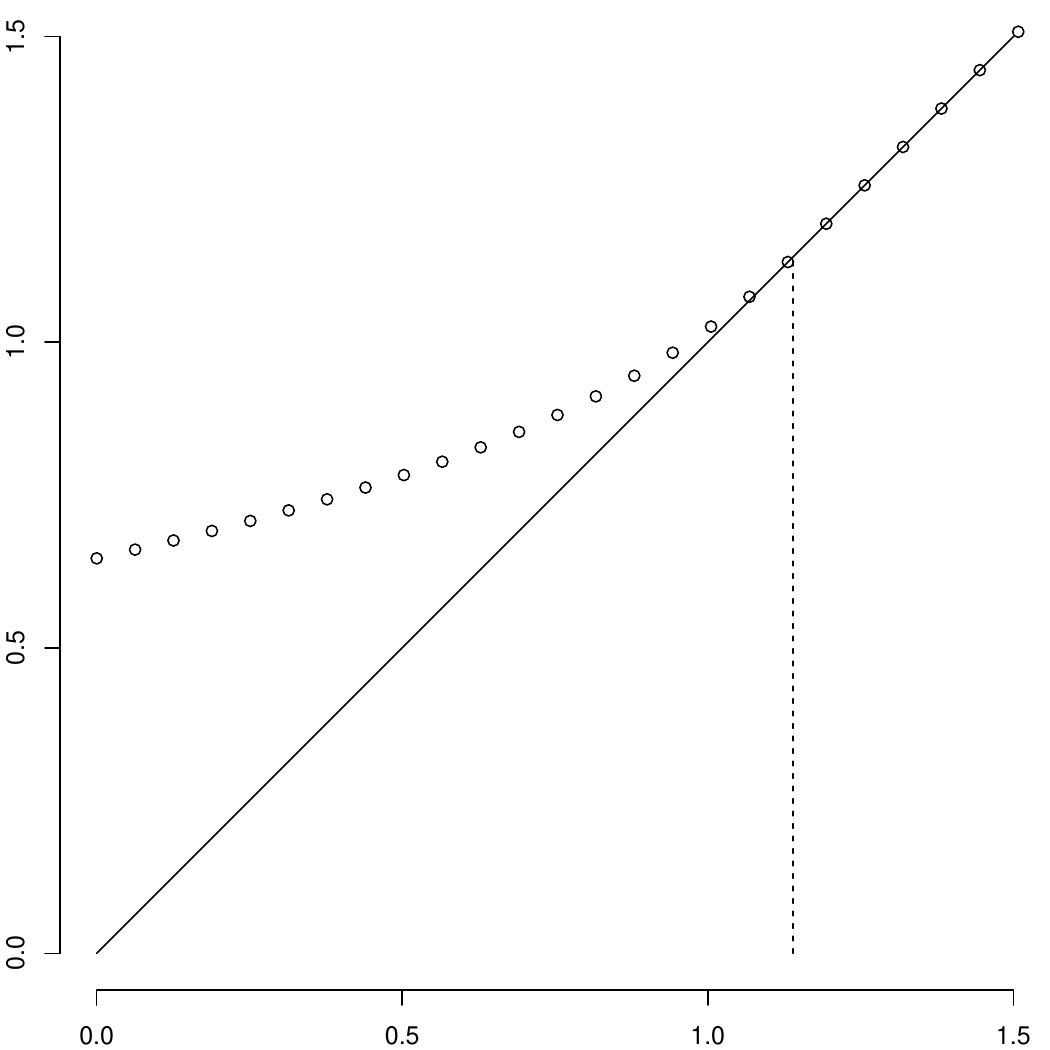}
\caption{\label{fig:ornstein-jumps} OSP for the Ornstein-Uhlenbeck with jumps: 
$g$ (continuous line), $V_1$ (circles) $x^*= 1.1442$.}
\end{center}
\end{figure}
\end{example}

\begin{remarks} (i) This example gives, up to our knowledge, the fist explicit solution to an optimal stopping problem for a process with jumps that is not a L\'evy process. (ii) We find interesting in this example the way in which the theoretical results, the Fourier methods and computational power gathers.
\end{remarks} 

\begin{example}[$\desc=\gamma=1$ and $\lambda=0$] 
In this example we consider the process $X$ already presented with parameter $\lambda=0$, i.e. with no jumps and the same reward function $\g(x)=x^+$. This problem was solved by \cite{taylor}.

We have 
\begin{align*}
\hGa(x,z)&=\left(e^{-\frac{1}{4}z^2}|z|^{-1}\right)\int_{0}^z  \frac{e^{i\zeta x}|\zeta|e^{\frac{1}{4}\zeta^2}}{\zeta} d\zeta\\
&=i \sqrt{\pi} e^{-\frac{1}{4}z^2}z^{-1} e^{x^2} \left(\erf\left(x-\frac{iz}{2}\right)-\erf(x)\right)
\end{align*}
As in the previous example, we solve numerically equation (ii) in Theorem \ref{teo:huntVerif} obtaining that $x^*\simeq 0.5939$.  \autoref{fig:ornstein} shows some points of the value function obtained numerically by the formula:
\bd
\Va(x)=\int_{x^*}^\infty \Ga(x,y)f(y) dy;
\ed
we also include in the plot the reward function (continuous line) to show that in the stopping region they coincide and also to verify that $\Va$ is a majorant of $g$ (hypothesis of \autoref{teo:huntVerif}). The obtained threshold is in accordance with the result obtained in \cite{taylor}.
\begin{figure}
\begin{center}
\includegraphics[scale=.25]{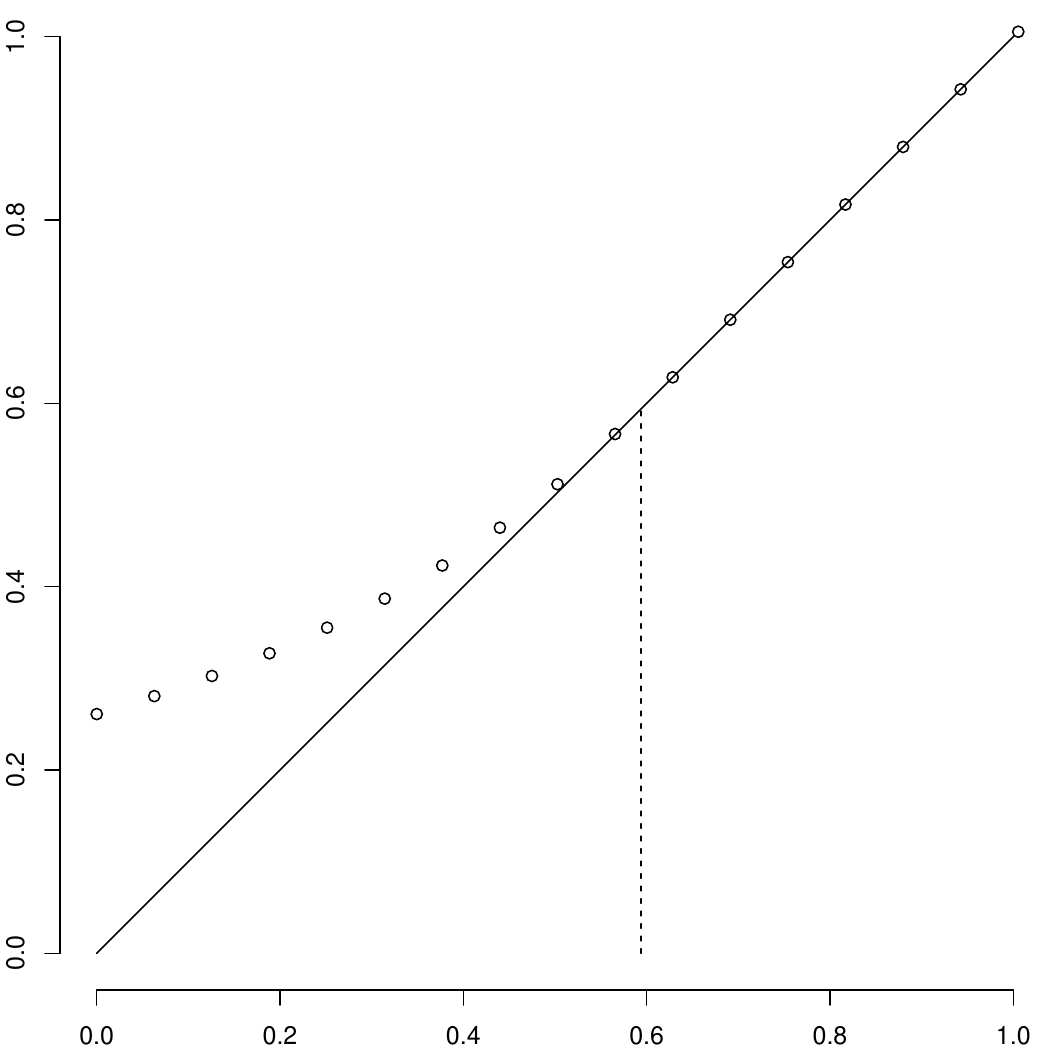}
\caption{\label{fig:ornstein} OSP for the Ornstein-Ulhenbeck process: $g$ (continuous line), $V_1$ (circles).}
\end{center}
\end{figure}
\end{example}


\end{document}